\documentclass[a4paper]{article}
\usepackage{latexsym}
\usepackage{amsmath,amssymb,amsfonts,graphicx}

\newtheorem{prethm}{{\bf Theorem}}

\newenvironment{thm}{\begin{prethm}{\hspace{-0.5
               em}{\bf.}}}{\end{prethm}}

\newtheorem{prepro}[prethm]{Proposition}

\newtheorem{prelem}[prethm]{Lemma}
\newenvironment{lem}{\begin{prelem}{\hspace{-0.5
               em}{\bf.}}}{\end{prelem}}

\newtheorem{precor}[prethm]{Corollary}
\newenvironment{cor}{\begin{precor}{\hspace{-0.5
               em}{\bf.}}}{\end{precor}}

\newtheorem{prerem}[prethm]{Remark}

\newtheorem{preconj}[prethm]{{\bf Conjecture}}
\newenvironment{conj}{\begin{preconj}{\hspace{-0.5
              em}{\bf.}}}{\end{preconj}}

\newtheorem{preexample}{{\bf Example}}

\newtheorem{preproof}{{\bf Proof.}}

\newenvironment{proof}[1]{\begin{preproof}{\rm
               #1}\hfill{$\Box$}}{\end{preproof}}

\newcommand{\noi}{\noindent}

\newcommand{\2}{\sqrt2}
\newcommand{\p}{\top}

\newcommand{\la}{\lambda}
\newcommand{\cl}{{\cal L}}
\newcommand{\h}{{\cal H}}
\newcommand{\s}{{\cal S}}
\newcommand{\D}{{\cal D}}
\newcommand{\R}{{\cal R}}
\newcommand{\q}{{\cal Q}}
\newcommand{\B}{{\rm BIBD}}
\newcommand{\spec}{{\rm spec}}

\title{Bipartite graphs with five eigenvalues and pseudo designs}
\author{\sc Ebrahim Ghorbani\\
{\small {Department of Mathematics, K.N. Toosi University of Technology,}}\\
{\small P.O. Box 16315-1618, Tehran, Iran}\\
 {\small { School of Mathematics, Institute
for Research in Fundamental Sciences (IPM),}} \\
 {\small { P.O. Box 19395-5746, Tehran, Iran}}\\
 {\tt\small e\_ghorbani@ipm.ir}}
\begin{document}
\maketitle

\begin{abstract}
A pseudo $(v,\, k,\, \la)$-design is a pair $(X, {\cal B})$ where $X$ is a $v$-set and ${\cal B}=\{B_1,\ldots,B_{v-1}\}$ is
a collection of $k$-subsets (blocks) of $X$
such that each two distinct $B_i, B_j$ intersect in $\la$ elements; and
$0\le\la <k \le v-1$. We use the notion of pseudo designs to characterize  graphs of order $n$ whose (adjacency) spectrum contains a zero and  $\pm\theta$ with multiplicity $(n-3)/2$ where $0<\theta\le\sqrt{2}$.
  Meanwhile, partial results confirming a conjecture of O. Marrero on characterization of pseudo $(v,\, k,\, \la)$-designs  are obtained.

\vspace{3mm}
\noindent {\em AMS Classification}:  05C50; 05B05; 05B30\\
\noindent{\em Keywords}: Spectrum of graph; pseudo design; BIBD; DS graph; Cospectral graphs; incidence graph; Subdivision of star
\end{abstract}

\section{Introduction}
 To study bipartite graphs with four/five  distinct (adjacency) eigenvalues, one needs to investigate combinatorial designs
  with two singular values (i.e. the matrix $NN^\p$  has only two positive eigenvalues where $N$ is the $(0,1)$-incidence matrix of the design).
  Recently, van Dam and Spence \cite{ds} studied bipartite graphs with four eigenvalues which are precisely the incidence graphs of designs withe two singular values with nonsingular and square $N$. These designs are called  uniform multiplicative designs, introduced by Ryser \cite{r2}. In \cite{ds2},   bipartite biregular graphs with five distinct eigenvalues were investigated.   These graphs correspond to designs with two singular values, constant block size and constant replication number.  These designs are called partial geometric designs first introduced by Bose (cf. \cite{bss}). Designs with few distinct singular values are also of interest from statistical point of view. R.A. Bailey (cf. \cite{c}) recently raised the question that which designs have  three eigenvalues. To be more specific, it was asked for which designs with constant block
size, constant replication, and incidence matrix $N$, does $NN^\p$ have three distinct eigenvalues.

In this paper, we continue this line by studying bipartite graphs with five eigenvalues where the second largest eigenvalue is relatively small. To be more precise, we characterize graphs with $n$ vertices whose spectrum contains
   $\{0,(\pm \theta)^{\frac{n-3}{2}}\}$ where $0<\theta\le\sqrt{2}$.
         The restriction to $\sqrt{2}$, comes from our limited knowledge of the corresponding designs. Having enough information of related designs, one can characterize graphs with larger $\theta$.  As it will be explained in the next section, it follows that  $\theta$ must be a square root of an integer.
         So  the next possible $\theta$ is $\sqrt{3}$.

          The graphs with $n$ vertices whose spectrum contains $(\pm\theta)^{\frac{n-2}{2}}$ with $0<\theta\le\sqrt{2}$ were already characterized by  van Dam and Spence  \cite{ds}. Note that the incidence graphs of symmetric $(v,\,k,\,\la)$-designs are precisely the regular graphs with the required property, and $\theta=\sqrt{k-\la}$.

 The graphs of the subject of the paper have a close connection with a family of combinatorial designs called pseudo designs.
  Therefore, we first study pseudo designs following Marrero \cite{m74, m77} and Woodall \cite{w}.
   Our investigation have some implications on a conjecture of Marrero on characterization of pseudo designs.
    We then make use of these results to determine the
   families of graphs whose spectrum contains $\{0,(\pm \theta)^{\frac{n-3}{2}}\}$.

By means of the spectral characterization of the aforementioned graphs, we find some new families of graphs  which are DS (i.e. determined by spectrum). Finding new families of DS graphs is one of the challenging and very active research subjects in spectral graph theory. For more about DS graphs see the surveys \cite{dh1,dh}.

\section{Preliminaries}

All the graphs that we consider in this paper are finite,
simple and undirected.
 The {\it order} of a graph $G$ is the number of vertices of $G$.
 By the eigenvalues of $G$ we mean
 those of its adjacency matrix. The {\em spectrum} of  $G$ is the multiset of eigenvalues of $G$.
 The {\em subdivision} of a graph $G$ is the graph obtained by inserting a new vertex
 on every edge of $G$. We denote by $\s_{2k+1}$ the subdivision of the star $K_{1,k}$.
The complete bipartite graph $K_{k,k}$ minus a perfect matching is denoted by $\cl_{k,k}$. We denote by $\h_{k,k+1}$ the resulting graph from
 adding a new vertex to $\cl_{k,k}$ and joining all
vertices of one part of it to the new vertex. The adjacency matrix of a bipartite graph $G$ can be rearranged so that it has the form
 $$\left(
    \begin{array}{cc}
      O & N \\
      N^\p & O \\
    \end{array}
  \right),$$
  where the zero matrices $O$ are square.
  We denote by $J_{r,\ell}$ the all 1 matrix with $r$ rows and $\ell$ columns, and by $J_r$ if it is square. When the order of the matrix is clear from the context, we drop the subscripts.
   The matrix $N$ is called the {\em bipartite adjacency matrix} of $G$.
  The bipartite graphs of order $n=2k+1$ with  bipartite adjacency matrices
$$\left(
  \begin{array}{cc}
  I_{k-3} &   J \\
    O&  \widetilde{I_3}
\end{array}\right)_{k\times(k+1)},~~~
\left(
  \begin{array}{cc}
  \widetilde{I_{k-3}} &   J \\
    O&  J_3-I_3
\end{array}\right)_{k\times(k+1)},$$
are denoted by $\R_n$ and $\q_n$, respectively, where $\widetilde{I_\ell}$ is the  matrix resulting from extending the identity matrix by an all 1 column vector,
i.e. $$\widetilde{I_\ell}=\left(\begin{array}{rl}I_\ell & {\bf1}_\ell \\\end{array}\right).$$
 The graphs $\R_{13}$ and $\q_{13}$ are depicted in Figure~\ref{rq}.
\begin{figure}
\centering \includegraphics[width=13cm]{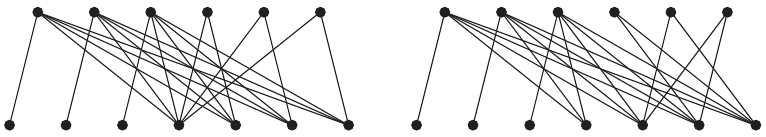}
 \caption{The graphs $\R_{13}$ (left) and $\q_{13}$ (right)}\label{rq}
  \end{figure}

A {\em combinatorial design} is a pair $(X, {\cal B})$ where $X$ is a set of points and ${\cal B}$
 is a collection of subsets of $X$, called blocks, together with an incidence relation between the points and the blocks.
A {\em balanced incomplete block design} BIBD$(b,\, v,\, r,\, k,\, \la)$ is a combinatorial design with $v$ points and $b$ blocks all of which have the same size $k$ and the incidence relation that  any 2-subset of $X$ is contained in exactly $\la$ blocks where $0\le\la$ and $k\le v-1$. It follows that every element of $X$ is contained
in a same number $r$ of blocks.
Necessary conditions for the existence of a BIBD$(b,\, v,\, r,\, k,\, \la)$ are
\begin{align}
  vr &= bk \label{bk},\\
  r(k-1)&=\la(v-1). \label{r(k-1)}
\end{align}
A BIBD$(b,\, v,\, r,\, k,\, \la)$ with $b=v$ (and so $r=k$) is called a symmetric $(v,\, k,\,\la)$-design. It is known that in a symmetric $(v,\, k,\,\la)$-design
two distinct blocks $B_i, B_j$  $(1\le i,j\le v)$ intersect in $\la$ elements.
A {\em pseudo $(v,\, k,\, \la)$-design} is a pair $(X, {\cal B})$ where $X$ is a $v$-set and ${\cal B}=\{B_1,\ldots,B_{v-1}\}$ is
a collection of $k$-subsets (blocks) of $X$ such that each two distinct $B_i, B_j$ intersect in $\la$ elements; and
$0\le\la <k \le v-1$.
Each combinatorial design is completely determined by its corresponding
{\em incidence matrix}; this is the $(0, 1)$-matrix $A=(a_{ij})$ whose rows and columns are indexed by the blocks and the points  of the design, respectively, where $a_{ij}=1$ if
$x_j\in B_i$ and $a_{ij}=0$ if
$x_j\not\in B_i$. The {\em incidence graph} of a design $\D$ is a bipartite graph such that its bipartite adjacency matrix is the incidence matrix of $\D$.

\noi{\bf Remark.} In order to avoid trivial cases, we assume that the designs considered in this paper(and so
their incidence graphs) are connected. Therefore, the Perron--Frobenius theorem (cf. \cite[p. 178]{gr})
can be applied. It follows that the largest singular value has
multiplicity one and a positive eigenvector. Another consequence is that if $N$ is the bipartite adjacency matrix of a connected bipartite graph of order $n$ with five distinct eigenvalues where the zero eigenvalue is of multiplicity 1, then the characteristic polynomial of $NN^\p$ is of the form $x(x^2-a)^{\frac{n-3}{2}}(x^2-b)$. As mentioned in \cite{ds}, it turns out that if $n>5$, then $a$ and $b$ must be integers. For $n=5$, there are only two such graphs with the following bipartite adjacency matrices:
$$\left(\begin{array}{ccc}1 & 1 & 1 \\0& 0 & 1 \\\end{array}\right)~~\hbox{and}~~
\left(\begin{array}{ccc}1 & 1 & 1 \\0& 1 & 1 \\\end{array}\right).$$
The spectra of these two graphs are $\{0,\pm\sqrt{2\pm\sqrt{2}}\}$ and $\{0,\pm\frac{1}{2}\sqrt{10\pm2\sqrt{17}}\}$, respectively.
Therefore, if $0<\theta\le\sqrt{2}$ and $\{0,(\pm \theta)^{\frac{n-3}{2}}\}$ is contained in the spectrum of a graph of order $n>5$, then $\theta=1$ or $\theta=\sqrt{2}$.

\section{Pseudo $(v,\,k,\,\la)$-designs}

Pseudo designs were studied by Marrero \cite{m74,m77} and Woodall \cite{w}.
Woodall used another terminology; he called pseudo designs {\em near-square $\la$-linked designs}.\footnote{A {\em square $\la$-linked designs} consists of  $v$ points and $v$
blocks such that each two distinct blocks intersect in $\la$ elements. This configuration ia called $\la$-design by Ryser \cite{r} if in addition there exist two blocks with different sizes.}
We follow the terminology of Marrero.

A pseudo $(v,\, k,\, \la)$-design is called
{\em primary} if $v\la\ne k^2$ and is called  {\em nonprimary} when $v\la=k^2$. It is shown that \cite{mb} in a nonprimary pseudo design,  $v=2k$. Thus a  pseudo   $(v,\, k,\, \la)$-design is nonprimary if and only if $v=4\la$ and $k=2\la$.

  The existence of a nonprimary pseudo $(v,\, k,\, \la)$-design is equivalent to existence of a Hadamard design:
\begin{thm} \label{noprim}{\rm(Marrero--Butson \cite{mb})}  The incidence matrix of a given pseudo
$(4\la, 2\la, \la)$-design can always be obtained from the incidence matrix $A$ of a symmetric
$(4\la-1, 2\la-1, \la-1)$-design by adjoining one column of all $1$'s to $A$
and then possibly complementing some rows of $A$.
\end{thm}
In the theorem, complementing a row means that $0$'s and $1$'s are interchanged in that row. For example, take the Fano plane which is the unique symmetric
$(7,\,3,\,1)$-design with points $\{1,\ldots,7\}$ and blocks
$\{124, 235, 346 , 457 , 156 , 267 ,137\}$. Now the theorem asserts that by adding a new point to all the blocks, namely 8,
and complementing any set of blocks we get a pseudo $(8,\,4,\,2)$-design. E.g. if we do this for the first block we have the pseudo design with blocks
 $\{3567, 2358, 3468, 4578, 1568, 2678,1378\}$.
 For primary pseudo designs we have the following:
\begin{thm}\label{4types} {\rm (Marrero \cite{m74, m77})} The incidence matrix $A$ of a primary pseudo $(v,\, k, \,\la)$-design $\D$ can be obtained from the
incidence matrix of a symmetric $(\bar v,\,\bar k, \bar\la)$-design whenever  $\D$ satisfies  one of the following
 arithmetical conditions on its parameters.
 \begin{itemize}
  \item[\rm(i)] If $(k-1)(k-2)=(\la-1)(v-2)$, then $A$ is obtained by adjoining a column of $1$'s to the incidence matrix of a symmetric $(v-1,\,k-1,\,\la-1)$-design.
  \item[\rm(ii)] If $k(k-1)=\la(v-2)$, then $A$ is obtained by adjoining a column of $0$'s to the incidence matrix of a symmetric  $(v,\,k,\,\la)$-design.
  \item[\rm(iii)] If $k(k-1)=\la(v-1)$, then $A$ is obtained from discarding a row from the incidence matrix of a symmetric $(v,\,k,\,\la)$-design.
  \item[\rm(iv)] If $k=2\la$, then $A$ is obtained from the incidence matrix  $B$ of a symmetric  $(v,\,k,\,\la)$-design as follows: a row is discarded from $B$ and then the $k'$
  columns of $B$ which had a $1$ in
the discarded row are complemented ($0$'s and $1$'s are interchanged in these
columns).
\end{itemize}
\end{thm}

It was conjectured by O. Marrero \cite{m74,m77} that given a primary pseudo $(v,\, k,\, \la)$-design,
  then `completion' or `embedding' between the given pseudo design and some symmetric $(\bar v,\, \bar k,\, \bar\la)$-design always is possible.
   In other words:
\begin{conj}\label{conj} {\rm (Marrero \cite{m74, m77})}  The parameters of a given primary pseudo $(v,\, k,\, \la)$-design  satisfy at least one of the four conditions of Theorem~\ref{4types}.
\end{conj}
He proved the validity of his conjecture for $\la=1$.

\begin{thm}\label{prim} {\rm (Marrero \cite{m77}, Woodall \cite{w})} Let $A$ be the incidence matrix of a given primary pseudo
$( v,\,  k,\, \la)$-design, so that $A$ has two distinct column sums $s_1$ and $s_2$. Let $y=
(k + \la( v - 2)- ks_2)/(s_1 - s_2)$, and let $f$ be the number of columns of $A$ having column
sum $s_1$. Then, after an appropriate permutation of the columns of $A$, it must be
possible to write $A=[M_{v-1,f}~ N_{v-1,v-f}]$, where $M$ is the incidence matrix of a
$\B(\bar b =  v - 1,\, \bar v = f,\, \bar r = s_1,\, \bar k = y,\, \bar\la = s_1-k +\la)$ and $N$ is the incidence
matrix of a $\B(\bar b=v-1,\, \bar v=v-f,\, \bar r=s_2,\, \bar k=k-y,\, \bar\la=s_2-k+\la)$. (Note that $f$ may take the values $1$ or $v-1$, too.)\end{thm}

In order to study the graphs of the subject of this paper, we need to characterize pseudo designs with $k-\la=1$ or $2$. To do so, we need the following lemma.
\begin{lem}\label{bibd} Let $\D$ be a $\B(b,\,v,\,r,\,k,\,\la)$.
\begin{itemize}
  \item[\rm(i)] If $r=\la+1$, then $\D$ is either the symmetric $(v,\,1,\,0)$-design or the symmetric $(v,\,v-1,\,v-2)$-design.
   \item[\rm(ii)] If $r=\la+2$, then $\D$ is one of the $\B(2v,\,v,\,2,\,1,\,0)$, $\B(6,\,4,\,3,\,2,\,1)$, $\B(6,\,3,\,3,\,2,\,2)$, the symmetric $(7,\,3,\,1)$-design, or the symmetric $(7,\,4,\,2)$-design.
  \end{itemize}
\end{lem}
\begin{proof}{The part (i) is straightforward. We prove  (ii).  First let $b>v$. So $\la+1=r-1\ge k$. By (\ref{r(k-1)}), $(\la+2)(k-1)=\la(v-1)$. If $\la=0$, then $r=2$ and $k=1$. So, by (\ref{bk}), $b=2v$ which means $\D$ is
$\B(2v,\,v,\,2,\,1,\,0)$. If $\la=1$, then $r=3$ and so $v=3k-2$. We have $k\ne1$ since otherwise $v=1$ which is impossible. Thus $k=2$. It turns out that
$\D$ is $\B(6,\,4,\,3,\,2,\,1)$. If $\la=2$, then $r=4$ and so $v=2k-1$. By (\ref{bk}), $bk=4(2k-1)$ from which it follows that $k=2$ and thus $b=6$.
 Hence $\D$ is $\B(6,\,3,\,3,\,2,\,2)$. Let $\la\ge3$. If $\la$ is odd, then $\la\mid k-1$ and thus $\la=k-1$.
 If $\la$ is even, then $\frac{\la}{2}\mid k-1$. Since $\la\ge k-1$, it follows that either $k-1=\la$ or $k-1=\frac{\la}{2}$, the latter is impossible due to (\ref{r(k-1)}).
 Therefore $(k+1)(k-1)=(k-1)(v-1)$ so $v=k+2$.
  On the other hand, $bk=(k+1)(k+2)$ which is impossible since $k\ge4$. Now let $b=v$. So $r=k=\la+2$. We have $(\la+2)(\la+1)=\la(v-1)$.
  Clearly $\la\ne0$. If $\la=1$, then $k=3$ and $v=7$. So $\D$ is the symmetric $(7,\,3,\,1)$-design. If $\la=2$, then $k=4$ and $v=7$. So $\D$ is the symmetric $(7,\,4,\,2)$-design.}
\end{proof}

\begin{thm}\label{pseudo} Let $\D$ be a pseudo $(v,\, k,\, \la)$-design.
\begin{itemize}
  \item[\rm(i)] If $k=\la+1$,  then $\D$ is obtained  from the symmetric $(v-1,\,1,\,0)$-design or the symmetric $(v-1,\, v-2,\,  v-3)$-design by either adding an isolated point or a point which belongs to all of the blocks.
  \item[\rm(ii)] If $k=\la+2$, then, up to isomorphism, $\D$ is one of the $\D_i=(\{1,\ldots,8\},{\cal B}_i)$, $i=1,2,3,4$, where
\begin{align*}
    {\cal B}_1=&\{1238,\,1458 ,\,1678 ,\,3568 ,\,2478 ,\,3468 ,\,2568 \},\\
     {\cal B}_2=&\{4567,\,1458 ,\,1678 ,\,3568 ,\,2478 ,\,3468 ,\,2568 \},\\
      {\cal B}_3=&\{4567,\,2367 ,\,1678 ,\,3568 ,\,2478 ,\,3468 ,\,2568 \},\\
       {\cal B}_4=&\{4567,\,2367 ,\,2345 ,\,3568 ,\,2478 ,\,3468 ,\,2568 \};
\end{align*}
or is obtained by omitting one block either from the unique symmetric $(7,\,4,\,2)$-design or the unique symmetric $(7,\,3,\,1)$-design.
\end{itemize}
 \end{thm}

\begin{proof}{ (i) First, let $\D$ be nonprimary. This is the case only if $\la=1$ and so $k=2$,  $v=4$. By Theorem~\ref{noprim}, $\D$ is obtained from a symmetric $(3,\,1,\,0)$-design
 by the technique described in Theorem~\ref{noprim}.  Applying this technique, it turns out that $\D$ is either the symmetric $(3,\,1,\,0)$-design with a point added to all of its blocks
 or the symmetric $(3,\,2,\,1)$-design with an extra isolated point. Now, let $\D$ be primary. In  view of Theorem~\ref{prim},  $\D$ is obtained by `pasting' two $\B(b_i=v-1,\, v_i,\, r_i,\, k_i,\, \la_i)$
 with $r_i=\la_i+1$. Keeping the notations of Theorem~\ref{prim}, we must have $f=1$. Thus $M$ is either the vector $\bf0$ or $\bf1$ and by Lemma~\ref{bibd},
  $N$ is the incidence matrix of either symmetric $(v-1,\,1,\,0)$-design or symmetric $(v-1,\, v-2,\,  v-3)$-design.

\noi(ii) First, let $\D$ be nonprimary. This is the case only if $\la=2$ and so $k=4$,  $v=8$. By Theorem~\ref{noprim}, $\D$ is obtained from the Fano plane by the technique
 described in Theorem~\ref{noprim}. Making use of the \textsf{Maple} procedure for checking graph isomorphism, it turns out that $\D$ is isomorphic to one of the pseudo
  designs $\D_1,\D_2,\D_3$, or $\D_4$.
Now, let $\D$ be primary.  Thus  $\D$ is obtained by `pasting' two $\B(\bar b_i=v-1,\,\bar v_i,\,\bar r_i,\,\bar k_i,\,\bar \la_i)$'s with $\bar r_i=\bar \la_i+2$ for $i=1,2$.
If $f=1$, then $M$ is either the vector $\bf0$ or $\bf1$ and by Lemma~\ref{bibd}, $N$ is the incidence matrix of
either symmetric $(7,\,3,\,1)$-design, or symmetric $(7,\,4,\,2)$-design. If $f\ge2$, then $M$ and $N$ must be chosen from the incidence matrices of $\B(6,\,4,\,3,\,2,\,1)$, $\B(6,\,3,\,3,\,2,\,2)$, or $\B(2\ell-2,\,\ell-1,\,2,\,1,\,0)$ for some $\ell\ge2$. Since $\bar v_1+\bar v_2=v=\bar b_1+1=\bar b_2+1$, the only possible choices for $M$ and $N$ are that either
 \begin{itemize}
   \item[1)] $M$ is the incidence matrix of $\B(6,\,4,\,3,\,2,\,1)$ and $N$ is that of $\B(6,\,3,\,3,\,2,\,2)$; or
   \item[2)] $M$ is the incidence matrix of $\B(6,\,4,\,3,\,2,\,1)$ and $N$ is that of $\B(2\ell-2,\,\ell-1,\,2,\,1,\,0)$ for $\ell=3$.
 \end{itemize}
 If 1) is the case, then $v=7$, $s_1=3$, $s_2=4$ and so $3k-5\la=y=2$ which  together with $k-\la=2$ give $\la=2$ and $k=4$. Now, $\D$ satisfies the conditions of parts (iii) and
 (iv) of Theorem~\ref{4types}.
 From  part (iii) it follows that $\D$
  is obtained from the symmetric $(7,\,4,\,2)$-design by omitting one of its blocks; and from part (iv)
     we see that  $\D=\{3567,\, 1467,\, 1257,
   \,1236,\, 2347, \,1345,\, 2456\}$ which is again the symmetric $(7,\,4,\,2)$-design with an omitted block.
 If 2) is the case, then $(v,\,k,\,\la)=(7,\,3,1\,)$ and by Theorem~\ref{4types}(iii), $\D$
  is obtained from the symmetric $(7,\,3,\,1)$-design by omitting one of its blocks.}
  \end{proof}

\begin{cor} Conjecture \ref{conj} holds for pseudo $(v,\,k,\,\la)$-designs with $k=\la+1$  or $k=\la+2$.
\end{cor}

\section{Graphs with many $\pm1$ eigenvalues}

 In this section we characterize all graphs of order $n$
whose spectrum contains a zero and $\pm1$ with multiplicity $(n-3)/2$.
We show that this family of graphs consists of $\s_{2k+1}$, $\h_{k,k+1}$, $\R_{2k+1}$, $\q_{2k+1}$ where $n=2k+1$ and two graphs of order 13.

We begin by determining the spectrum of $\s_n$, $\cl_{k,k}$, $\h_{k,k+1}$, $\R_n$, and $\q_n$.

\begin{lem}\label{spec}
\begin{itemize}
  \item[\rm(i)] $\spec(\s_{2k+1})=\left\{\pm\sqrt{k+1},\, 0,\, (\pm1)^{k-1}\right\}$,
  \item[\rm(ii)] $\spec(\cl_{k,k})=\left\{\pm(k-1),\, (\pm1)^{k-1}\right\}$,
  \item[\rm(iii)] $\spec(\h_{k,k+1})=\left\{\pm\sqrt{k^2-k+1},\,0 ,\, (\pm1)^{k-1}\right\}$, for $k\ge2$,
  \item[\rm(iv)] $\spec(\R_{2k+1})=\spec(\q_{2k+1})=\left\{\pm2\sqrt{k-2},\,0 ,\, (\pm1)^{k-1}\right\}$, for $k\ge3$.
\end{itemize}
\end{lem}
\begin{proof}{If one deletes the vertex of maximum degree from $\s_{2k+1}$, what remain are $k$ copies of  $K_2$. Thus, by interlacing, the spectrum of $\s_{2k+1}$ contains $\pm1$ of multiplicity at least $k-1$. Since $\s_{2k+1}$ is a bipartite graph of an odd order, it has a zero eigenvalue. Let $\pm\theta$ be the remaining eigenvalues.  As the sum of squares of eigenvalues of a graph is twice the number of edges, we have $2\theta^2+2k-2=4k$ implying $\theta=\sqrt{k+1}$. The spectrum of $\cl_{k,k}$ is easily obtained since it has an adjacency matrix of the form
$$\left(
  \begin{array}{cc}
    O & J_k-I_k \\
    J_k-I_k & O\\
  \end{array}
\right).$$
 The graph $\h_{k,k+1}$ possesses an `equitable partition' with three cells in which each cell consists of the vertices with equal degree. (See \cite[pp. 195--198]{gr} for more information on equitable partitions.) The adjacency matrix of the corresponding quotient is $$B=\left(
    \begin{array}{ccc}
      0 & k-1 & 1 \\
      k-1 & 0 & 0 \\
      k & 0 & 0 \\
    \end{array}
  \right),
$$
with eigenvalues $\pm\sqrt{k^2-k+1},\,0$.
   Besides these three eigenvalues, by interlacing, $\h_{k,k+1}$ has $\pm1$ eigenvalues of multiplicity at least $k-2$. Let $\pm\theta$ be the remaining eigenvalues. Thus,
$2(k^2-k+1)+2(k-2)+2\theta^2=2k^2$, which implies $\theta=1$. If $N$ is the bipartite adjacency matrix of either $\R_{2k+1}$
or $\q_{2k+1}$, then
$$NN^\p-I=\left(
    \begin{array}{cc}
      4J_{k-3} & 2J \\
      2J^\p& J_3
    \end{array}
  \right).$$
Thus $NN^\p-I$ is of rank one and so both $\spec(\R_{2k+1})$ and $\spec(\q_{2k+1})$ contain $\{0, (\pm1)^{k-1}\}$. For the two remaining eigenvalues $\pm\theta$
we have the equation $2(k-1)+2\theta^2=10(k-3)+12$ and so $\theta=2\sqrt{k-2}$.
}\end{proof}

Before treating the graphs of the subject of this section, we deal with the graphs of order $n$
whose spectrum contains $(\pm 1)^{\frac{n-2}{2}}$. If such a graph is regular, then it is easily follows that $G$ must be $K_{\frac{n}{2},{\frac{n}{2}}}$ minus a perfect matching. If it is regular, by \cite[Proposition~8]{ds}, $G$ is either the graph $G_1$ or $G_2$ of Figure~\ref{fig}. So we have the following:

\begin{thm} {\rm(van Dam--Spence \cite{ds})} \label{L_n} Let $G$ be a connected graph of order $n$.
 If the spectrum of $G$ contains $(\pm 1)^{\frac{n-2}{2}}$,
  then $G$ is either $\cl_{\frac{n}{2},{\frac{n}{2}}}$ or the graph $G_1$ or $G_2$ of Figure~\ref{fig}.
\end{thm}

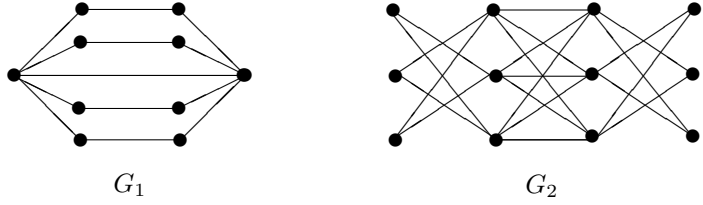
\begin{figure}\begin{center}
\setlength{\unitlength}{1700sp}
\begin{picture}(3360,1920)(3053,-3647)
\put(4049,-1739){\circle*{200}}
\put(5441,-1739){\circle*{200}}
\put(6401,-2699){\circle*{200}}
\put(5441,-2219){\circle*{200}}
\put(4025,-2219){\circle*{200}}
\put(3065,-2699){\circle*{200}}
\put(5441,-3179){\circle*{200}}
\put(4001,-3179){\circle*{200}}
\put(4025,-3635){\circle*{200}}
\put(5465,-3635){\circle*{200}}
\put(6377,-2699){\circle*{200}}
\put(3089,-2699){\line( 1, 1){960}}
\put(4049,-1739){\line( 1, 0){1392}}
\put(5441,-1739){\line( 1,-1){960}}
\put(6401,-2699){\line(-2, 1){960}}
\put(5441,-2219){\line(-1, 0){1416}}
\put(4025,-2219){\line(-2,-1){960}}
\put(3065,-2699){\line( 1, 0){3336}}
\put(6401,-2699){\line(-2,-1){960}}
\put(5441,-3179){\line(-1, 0){1440}}
\put(4001,-3179){\line(-2, 1){864}}
\put(3137,-2747){\line( 1,-1){888}}
\put(4025,-3635){\line( 1, 0){1440}}
\put(5465,-3635){\line( 1, 1){912}}
\put(6377,-2699){\line(-1, 0){3288}}
\put(4500,-4400){$G_1$}
\end{picture}
\hspace{3cm}
\begin{picture}(1487,1924)(4214,-3436)
\put(4239,-1536){\circle*{200}}
\put(5689,-1524){\circle*{200}}
\put(4276,-2499){\circle*{200}}
\put(5664,-2461){\circle*{200}}
\put(4276,-3424){\circle*{200}}
\put(5664,-3374){\circle*{200}}
\put(4239,-1536){\line( 1, 0){1450}}
\put(5689,-1524){\line(-3,-2){1428.231}}
\put(4276,-2499){\line( 1, 0){1388}}
\put(5664,-2461){\line(-3,-2){1405.385}}
\put(4276,-3424){\line( 1, 0){1388}}
\put(5664,-3386){\line(-4, 5){1463.610}}
\put(4226,-1536){\line( 3,-2){1413.461}}
\put(5664,-1536){\line(-3,-4){1405.920}}
\put(4276,-3424){\line( 1, 0){1388}}
\put(5664,-3374){\line(-3, 2){1370.769}}
\put(7139,-1524){\circle*{200}}
\put(7114,-2461){\circle*{200}}
\put(7114,-3374){\circle*{200}}
\put(7089,-1524){\line(-3,-2){1428.231}}
\put(7114,-2461){\line(-3,-2){1405.385}}
\put(7114,-3386){\line(-4, 5){1463.610}}
\put(5726,-1536){\line( 3,-2){1413.461}}
\put(7114,-1536){\line(-3,-4){1405.920}}
\put(7114,-3374){\line(-3, 2){1370.769}}
\put(2789,-1536){\circle*{200}}
\put(2826,-3424){\circle*{200}}
\put(2826,-2499){\circle*{200}}
\put(4239,-1524){\line(-3,-2){1428.231}}
\put(4214,-2461){\line(-3,-2){1405.385}}
\put(4214,-3386){\line(-4, 5){1463.610}}
\put(2776,-1536){\line( 3,-2){1413.461}}
\put(4214,-1536){\line(-3,-4){1405.920}}
\put(4214,-3374){\line(-3, 2){1370.769}}
\put(4700,-4200){$G_2$}
\end{picture}
\end{center}
\caption{\small The only nonregular graphs of order $n$ whose spectrum contain $(\pm1)^{\frac{n-2}{2}}$}
\label{fig}
\end{figure}

\begin{thm}\label{S_n} Let $G$ be a connected graph of order $n$. If the spectrum of $G$ contains $\{0,\,(\pm1)^\frac{n-3}{2}\}$, then $G$ is one of the graphs $\s_n$, $\R_n$, $\q_n$, $\h_{\frac{n-1}{2},\frac{n+1}{2}}$, $G_3$, or $G_4$ of Figure~\ref{g3g4}.
\end{thm}
\begin{proof}{From the spectrum of $G$ it is obvious that $G$ is bipartite of order $n=2k+1$. Let $N$ be the $r\times s$ bipartite adjacency matrix of $G$ where $r+s=2k+1$ and $r\le s$.
Considering the rank of the adjacency matrix of $G$, we have ${\rm rank}(N)=k$. This implies that $r=k$ and $s=k+1$.
So $NN^\p$ is nonsingular with two distinct eigenvalues $1,\theta$, say. Since the multiplicity of eigenvalue $1$ is $k-1$, $NN^\p-I$ is a rank one matrix, and by the Perron--Frobenius theorem, one may choose a positive eigenvector ${\bf x}=(x_1,\ldots,x_k)$ of $NN^\p$ for $\theta$ so that
\begin{equation}
  NN^\p = I+{\bf x}^\p{\bf x}.\label{nnt}
\end{equation}
If the vertices corresponding to the rows of $N$ are labeled $\{1,\ldots,k\}$, from (\ref{nnt}) it follows that
\begin{align}
  d_i = &~ 1+x_i^2,\label{di} \\
  d_{ij} = &~ x_ix_j,\label{dij}
\end{align}
where $d_i$ and $d_{ij}$, for $i,j=1,\ldots,k$, are the degree of the vertex $i$ and the number of common neighbors of the vertices $i,j$, respectively. It turns out that ${\bf x}=\sqrt{\delta}{\bf w}$, where ${\bf w}=(w_1,\ldots,w_k)$ is a positive integer vector and $\delta$ is a square-free integer.

 First let $d_i=d$ for $i=1,\ldots,k$. By (\ref{di}) and (\ref{dij}), $d_{ij}=d-1$, for every $i,j$.
This means that $N$ is the incidence matrix of a pseudo $(k,\,d,\,d-1)$-design. Therefore from Theorem~\ref{pseudo}
it follows that  $G$ is either $\s_{2k+1}$ or $\h_{k,k+1}$.
\begin{figure}
 \centering \includegraphics[width=10cm]{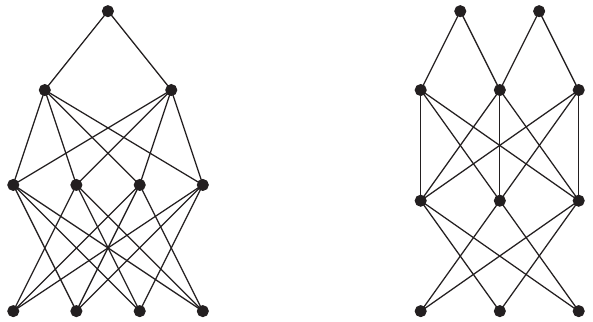}\\
 \caption{The graphs $G_3$ and $G_4$}\label{g3g4}
  \end{figure}

Now let $d_i>d_j$ for some $i,j$. Thus $w_i\ge w_j+1$, and
 $$   d_j\ge d_{ij}=\delta w_iw_j\ge\delta w_j^2+\delta w_j\ge\delta w_j^2+1=d_j.$$
So one must have the equality in all the above inequalities which
implies $\delta=w_j=1$, $w_i=2$,
 and so $d_{ij}=d_j=2$, $d_i=5$. Therefore, the vertices of $G$ corresponding to the rows of $N$ are of degree either 2 or 5 and any vertex of degree 2
 has all of its neighbors in common with any vertex of degree 5. It thus follows that $N$ can be rearranged so that
 \begin{equation}\label{n}
   N=\left(
    \begin{array}{cc}
      N_1 & J \\
      O & N_2
    \end{array}
  \right),
 \end{equation}
in which $N_1$ and $N_2$ correspond to the vertices of degree 5 and 2, respectively.
Suppose that $N_1$ and $N_2$ are $k_1\times\ell_1$ and $k_2\times\ell_2$, respectively. With the above rearrangement, ${\bf x}=(2,\ldots,2,1,\ldots,1)$ with $k_1$ $2$'s and $k_2$ $1$'s.
So
$${\bf x}^\p{\bf x}=\left(
    \begin{array}{cc}
      4J_{k_1,k_1} & 2J_{k_1,k_2} \\
      2J_{k_2,k_1}& J_{k_2,k_2} \\
    \end{array}
  \right).$$
  If $\ell_1=0$,  then
  $$NN^\p=\left(\begin{array}{c} J \\ N_2\end{array}\right)\left(\begin{array}{cc} J^\p & N_2^\p \end{array}\right)
  =\left(\begin{array}{cc}
    \ell_2J_{k_1,k_1} & 2J_{k_1,k_2} \\
    2J_{k_2,k_1}& J_{k_2,k_2}+I\\
 \end{array}
 \right).$$
  In view of (\ref{nnt}), we must have $k_1=1$ which in turn implies that $\ell_2=5$ and $k_2=3$.
  So the graph $G$ consists of three vertices of degree 2, say $v_1,v_2,v_3$ and one vertex of degree 5 in one part and five other vertices in the other part such that each of these five latter vertices is adjacent to at least one of $v_1,v_2,v_3$. On the other hand, since  $N_2N_2^\p=J+I$, each pair of $v_1,v_2,v_3$ have a common neighbor, which is not possible.
   Therefore, $\ell_1>0$. By inspecting $NN^\p$ and (\ref{nnt}), we have
\begin{eqnarray}
  N_1N_1^\p&=&I+(4-\ell_2)J,\label{n1} \\
  N_2N_2^\p&=&I+J\label{n2},
\end{eqnarray}
and moreover, one of $N_1$ or $N_2$ must be square since otherwise from (\ref{n}) it is clear that ${\rm rank}(N)\le k-2$ which is a contradiction.

First let $N_1$ be square. Thus $N_1$ is the incidence matrix of a symmetric $(k_1,3-\ell_2,4-\ell_2)$-design $\D_1$ and $N_2$ is the incidence matrix of a pseudo $(\ell_2,\,2,\,1)$-design. Therefore, by Theorem~\ref{bibd}, $\D_1$ is either the symmetric $(k_1,\,1,\,0)$-design or the  symmetric $(k_1,\,k_1-1,\,k_1-2)$-design. If $\D_1$ is the symmetric $(k_1,\,1,\,0)$-design, then $\ell_2=4$, and so by Theorem~\ref{pseudo}, $\D_2$ is obtained from the symmetric $(3,\,1,\,0)$-design by adding a new point to all of its blocks. So we find that $N_1=I_{k_1}$ and
$N_2=\widetilde{I_3}$.
 Therefore, $G$ is $\R_n$.
If $\D_2$ is the symmetric $(k_1,\,k_1-1,\,k_1-2)$-design, then we must have $4-\ell_2=k_1-2$ and so $k_2=5-k_1$.
As $I+(4-\ell_2)J$ is a positive semidefinite matrix, $\ell_2\le4$. As $k_2\ge1$, we have also $\ell_2\ge2$.  If $\ell_2=4$, then $G$ is $\R_{11}$. If $\ell_2=2,3$, then $(k_1,k_2)=(4,1)$ or $(3,2)$ from which it follows that
$G$ is either $G_3$ or $G_4$, respectively.

Now, let $N_2$ be square. From Theorem~\ref{bibd} it follows that $N_2$ is the incidence matrix of the symmetric $(k_2,\,k_2-1,\,k_2-2)$-design. By (\ref{n2}), $k_2-2=1$. Thus $N_2=J_3-I_3$, $\ell_2=3$, and $N_1N_1^\p=I+J$. So $N_1$ is the incidence matrix of a pseudo design which by Theorem~\ref{pseudo} obtained in one of the following three ways: 1) From a symmetric $(k_1,\,1,\,0)$-design by adding an extra point to all the blocks, i.e. $N_1=\widetilde{I_{k_1}}$
which means that $G$ is the graph $\q_n$. 2) From a symmetric $(k_1,\,k_1-1,\,k_1-2)$-design  by adding an extra point to all the blocks, so $k_1-2=0$  which implies $G$ to be $\q_5$. 3)  From a symmetric $(k_1,\,k_1-1,\,k_1-2)$-design  by adding an isolated point which is impossible as this makes $G$ disconnected.
}\end{proof}

In the rest of this section we determine the spectral characterization of the graphs discussed so far. We begin by $\cl_{k,k}$ which is readily seen that it is DS as it is
the only $(k-1)$-regular bipartite  graph of order $2k$.

For later use we need to mention the spectrum of the graphs $G_1,G_2,G_3$ and $G_4$:
$$\begin{array}{ll}
   \spec(G_1)=\left\{\pm3,\,  (\pm1)^4\right\},&  \spec(G_2)=\left\{\pm4,\,  (\pm1)^5\right\},\\
   \spec(G_3)=\left\{\pm3\sqrt2,\, 0,\,  (\pm1)^4\right\},  &\spec(G_4)=\left\{\pm\sqrt{15},\, 0,\, (\pm1)^4\right\}.
\end{array}$$

\begin{cor} The graph $\h_{k,k+1}$ is {\rm DS} for $k\ge2$.
\end{cor}
\begin{proof}{Any cospectral mate $H$ of $\h_{k,k+1}$ for $k\ge2$ must have one of the graphs of Theorems~\ref{L_n} and \ref{S_n} as a connected component. Nothing that $k^2-k+1$ is always odd and never (unless $k=1$) a perfect square, $H$ cannot have one of
$\cl_{t,t}$, $\R_{2t+1}$, $\q_{2t+1}$, for any $t$, or $G_1, G_2, G_3$ as a component. Considering the number of edges, $\s_{2t+1}$, for any $t$, cannot be a component of $H$. The same is for $G_4$ as the equation $k^2-k+1=15$ has no integral solution.
}\end{proof}

 The graphs $\s_n$ belong to a family of trees called {\em starlike trees} (trees with only one vertex of degree larger than 2).
 In \cite{dh}, it was asked to determine which starlike trees are DS.
Partial results are obtained by several authors (cf. \cite{dh}). For this specific starlike trees,  Brouwer \cite{b} showed that the graphs $\s_n$ are DS among trees.
Here, we completely determine the spectral characterization of the graphs $\s_n$. The proof is the same as proof of the above corollary.

\begin{cor}\label{SDS} The graph $\s_{2k+1}$ is {\rm DS} if $k\not\in S$, where
$$S=\{4\ell+3\mid\ell\in\mathbb{N}\}\cup\{\ell^2-1\mid\ell\in\mathbb{N}\}\cup\{\ell^2-\ell\mid\ell\in\mathbb{N}\}\cup\{14,17\}.$$
Moreover, for $k\in S$ we have
\begin{itemize}
  \item $\s_{17}$ has exactly two cospectral mates which are $\cl_{4,4}\cup4K_2\cup K_1$ and $G_1\cup3K_2\cup K_1$; 
  \item $\s_{29}$ has exactly one cospectral mate which is $G_4\cup9K_2$;
  \item $\s_{31}$ has exactly four cospectral mates which are $G_2\cup9K_2\cup K_1$, $\cl_{5,5}\cup10K_2\cup K_1$, $\R_{13}\cup9K_2$, and $\q_{13}\cup9K_2$; 
  \item $\s_{35}$ has exactly one cospectral mates which is $G_3\cup12K_2$;
  \item if $k=4\ell+3$ and $\ell$ is not an integer of the form $t^2-1$, then $\s_{2k+1}$ has exactly two cospectral mates which are  $\R_{2\ell+7}\cup3\ell K_2$ and $\q_{2\ell+7}\cup3\ell K_2$;
   \item if $k=\ell^2-1$, $\ell=2t$, and $k\ne15$, then $\s_{2k+1}$ has exactly three cospectral mates which are $\cl_{\ell+1,\ell+1}\cup(k-\ell-1)K_2\cup K_1$, $\R_{2t^2+5}\cup3(t^2-1)K_2$, and $\q_{2t^2+5}\cup3(t^2-1)K_2$;
  \item if $k=\ell^2-1$, $\ell$ is odd, and $k\ne8$, then $\s_{2k+1}$ has exactly one cospectral mate which is $\cl_{\ell+1,\ell+1}\cup(k-\ell-1)K_2\cup K_1$;
  \item if $k=\ell^2-\ell$, then $\s_{2k+1}$ has exactly one cospectral mate which is $\h_{\ell,\ell+1}\cup(k-\ell)K_2$.
\end{itemize}
\end{cor}

\begin{cor}  The graph $\R_7$,
has exactly three cospectral mates, namely $\q_7$,  $S_7$, and $\cl_{3,3}\cup K_1$.
If $k=\ell^2+2$, for some $\ell\ge2$, then the graph $\R_{2k+1}$ has exactly two cospectral mates, namely $\q_{2k+1}$ and $\cl_{2\ell+1,2\ell+1}\cup(\ell-1)^2K_2\cup K_1$.
 For other values of $k\ge3$, the only cospectral mate of  $\R_{2k+1}$ is  $\q_{2k+1}$.
\end{cor}

In addition to that the graphs $\R_n$ and $\q_n$ are cospectral, they are related through switching.
We first recall the Seidel switching.
Let $G$ be a graph with  vertex set $V$, and  $X \subseteq V$.
From $G$ we obtain a new graph by leaving adjacency and non-adjacency inside $X$ and
$V \setminus X$ as it was, and interchanging adjacency and non-adjacency between $X$ and
$V\setminus X$ . This new graph is said to be obtained by Seidel switching with respect to the set $X$.
Now, in the graph $\R_n$, let $X$ be the set of four vertices corresponding to the columns of the submatrix $J$ in the bipartite adjacency matrix of $\R_n$. If we apply the Seidel switching on $\R_n$ with respect to $X$ we obtain $\q_n$.

\section{Graphs with many $\pm\2$ eigenvalues}

In this section we characterize all graphs of order $n$
whose spectrum contains a zero and $\pm\sqrt{2}$ with multiplicity $(n-3)/2$. It turns out that, up to isomorphism, there are exactly six such graphs, all of which are obtained in some way from the
 Fano plane.

We start with graphs of order $n$
  whose spectrum contain $(\pm\2)^\frac{n-2}{2}$. Let $N$ be the  $\frac{n}{2}\times \frac{n}{2}$ bipartite adjacency matrix of $G$.
 If $G$ is regular of degree $k$, say, then $NN^\p=(k-2)I+2J$ which means that $N$ is the incidence matrix of a $(n/2,\,k,\,k-2)$-design. Hence, by Lemma~\ref{bibd}, $N$ is  the incidence matrix of either the Fano plane or the complement of the Fano plane. The nonregular ones are characterized in \cite[Proposition~9]{ds}.

\begin{thm} {\rm(van Dam--Spence \cite{ds})} Let $G$ be a connected graph of order $n$.
  If the spectrum of $G$ contains $(\pm\2)^\frac{n-2}{2}$,
  then the bipartite adjacency matrix of $G$ is one of the following:
  \begin{itemize}
    \item[\rm(i)] incidence matrix of  the Fano plane (i.e. $G$ is the Heawood graph);
   \item[\rm(ii)] incidence matrix of the complement of the Fano plane;
  \item[\rm(iii)] \begin{equation}\label{N1N2}
   \left(
    \begin{array}{cc}
     N_1 & J_7 \\
    O_7  & N_2  \\
    \end{array}
  \right)~~\hbox{or}~~~
  \left(
    \begin{array}{ccc}
     1 & {\bf1}^\p & {\bf1}^\p \\
     {\bf1}  & I_5 & I_5 \\
      {\bf1} & I_5 & J_5-I_5
    \end{array}
  \right),
  \end{equation}
    where $N_1$ and $N_2$ are the incidence matrices of the Fano plane and the symmetric $(7,\,4,\,2)$-design, respectively.
  \end{itemize}
\end{thm}

\begin{thm} Let $G$ be a connected graph of order $n$. If the spectrum of $G$ contains $\{0,(\pm\2)^\frac{n-3}{2}\}$, then $G$ is incidence graph of one of
\begin{itemize}
    \item[\rm(i)]  the pseudo $(7,\,3,\,1)$-design;
   \item[\rm(ii)] the pseudo $(7,\,4,\,2)$-design; or
  \item[\rm(iii)] $\D_1,\ldots,\D_4$  of Theorem~\ref{pseudo}.
 \end{itemize}
\end{thm}
\begin{proof}{From the spectrum of $G$ it is clear that $G$ is bipartite of order $n=2k+1$. Let $N$ and ${\bf x}=(x_1,\ldots,x_k)$ be the same as in the proof of Theorem~\ref{S_n}. Thus
$NN^\p =2I+{\bf x}^\p{\bf x}$ and so
\begin{align}
  d_i =&~ 2+x_i^2,\label{di2} \\
  d_{ij} = &~ x_ix_j.\label{dij2}
\end{align}
 Again we have ${\bf x}=\sqrt{\delta}{\bf w}$, where ${\bf w}=(w_1,\ldots,w_k)$ is a positive integer vector and $\delta$ is a square-free integer.

First assume that there exist some $i,j$ such that $d_i>d_j$. Then
 $w_i\ge w_j+1$, and
$$d_j\ge d_{ij}=\delta w_iw_j\ge\delta w_j^2+\delta w_j.$$
If $\delta w_j=1$, then $\delta=w_j=1$, and so $d_j=2$ which is impossible. So $\delta w_j=2$ and equalities must occur in all the above inequalities.
Hence two cases may occur: 1) $\delta=1$ and $w_j=2$ which implies $d_j=d_{ij}=6$ and $d_i=11$; or 2) $\delta=2$ and $w_j=1$ which implies $d_j=d_{ij}=4$ and $d_i=10$.
Again, like the proof of Theorem~\ref{S_n}, $N$ can be rearranged so that
$$N=\left(
    \begin{array}{cc}
      N_1 & J \\
      O & N_2
    \end{array}
  \right),$$
in which $N_2$, with $k_2$ rows, say, correspond to the vertices with smaller degrees.
 Then in the same manner as the proof of Theorem~\ref{S_n}, we see that $N_2N_2^\p=2I+\ell J$, where $\ell$ is either 16 or 36. As $N_2$ is either $k_2\times k_2$ or $k_2\times(k_2+1)$, it is the incidence matrix of either a symmetric $(k_2,\,\ell+2,\,\ell)$-design or a pseudo $(k_2,\,\ell+2,\,\ell)$-design. Such designs do not exist by Lemma~\ref{bibd} and Theorem~\ref{pseudo}.

 Therefore, $d_i=d$ for $i=1,\ldots,k$. By (\ref{di2}) and (\ref{dij2}), $d_{ij}=d-2$, for every $i,j$. So $N$ is the incidence matrix of a pseudo $(k,\,d,\,d-2)$-design. Thus the result follows from Theorem~\ref{pseudo}.}
\end{proof}

 \noindent{\bf Acknowledgements.} The research of the author was in part supported by a grant from IPM (No. 90050117). The author is grateful to Osvaldo Marrero for supplying the literature of pseudo designs, to Jack Koolen for carefully reading the manuscript, and to the referees whose helpful comments improved the presentation of the paper.

\end{document}